\RequirePackage[final]{graphicx}
\documentclass[oneside, dvipsnames, svgnames, table, final]{amsart} %
\usepackage{etoolbox}
\usepackage{ifdraft}

\usepackage{xcolor}

\usepackage[colorlinks=true, citecolor=green!50!black, linkcolor=red!50!black, pagebackref=false, final]{hyperref}

\ifdraft{\usepackage[color,notref, notcite]{showkeys}}{}
\usepackage[letterpaper]{geometry}
\usepackage{xargs}
\usepackage{amsmath, amsthm, mathrsfs, amssymb}
\usepackage{tikz-cd} 
\usepackage[final]{graphicx}
\usepackage{enumitem}
\usepackage{bm}

\usepackage[T1]{fontenc}

\newcommand{\mathbx}[1]{\mathbb{#1}}

\usepackage[colorinlistoftodos,prependcaption,textsize=tiny]{todonotes} 

\usepackage{fourier}
\usepackage[foot]{amsaddr}

\theoremstyle{plain}
\ifcscounter{chapter}{%
        \newtheorem{thm}{Theorem}[chapter]%
        }%
        {%
        \newtheorem{thm}{Theorem}[section]%
        }
\newtheorem{theorem}[thm]{Theorem}

\newtheorem{lemma}[thm]{Lemma}

\newtheorem{proposition}[thm]{Proposition}

\newtheorem*{claim*}{Claim} 
\newtheorem*{thm*}{Theorem}
\newtheorem*{theorem*}{Theorem}

\theoremstyle{definition}

\newtheorem{definition}[thm]{Definition}

\theoremstyle{remark}

\newtheorem{remark}[thm]{Remark}

\newtheorem{example}[thm]{Example}

\newcommand{\C}{\mathbx{C}}

\newcommand{\Q}{\mathbx{Q}}

\newcommand{\Z}{\mathbx{Z}}

\newcommand{\ZZ}{\Z}

\newcommand{\Gl}{\operatorname{GL}}

\providecommand{\deg}{\operatorname{deg}}

\newcommand{\GL}{\Gl}

\renewcommand{\deg}{\operatorname{deg}}

\renewcommand{\vec}[1]{\mathbf{#1}}

\newcommand{\Orb}{\operatorname{Orb}}
\newcommand{\Stab}{\operatorname{Stab}}
\newcommand{\lcm}{\operatorname{lcm}}

\usepackage[backref=true,giveninits=true,backend=bibtex,style=numeric,url=false]{biblatex}
\author{Matthew Bolan}
\address[M.~Bolan]{Department of Mathematics, University of Toronto\\
Bahen Centre, Room 6290\\
40 St. George St., Toronto, ON, M5S 2E4\\
Canada.}
\email[M.~Bolan]{matthew.bolan@mail.utoronto.ca}
\author{Ben Williams}
 \address[B.~Williams]{Department of Mathematics, the University of British Columbia\\
   1984 Mathematics Rd \\
   Vancouver BC V6T 1Z2\\
 Canada.}
\email[B.~Williams]{tbjw@math.ubc.ca}
\thanks{
  B.~Williams acknowledges the support of the Natural Sciences and Engineering Research Council of Canada (NSERC), RGPIN-2021-02603 and of the Pacific Institute for the Mathematical Sciences (PIMS) through CRG41.
}

\begin{document}
\title{The degrees of irreducible factors of binomials and multiplicativity of the $n$-Hartley condition}

\begin{abstract}
  We prove that the degrees of irreducible factors of a binomial $t^n-\alpha$ are divisible by the least such degree. As a consequence, we deduce that for relatively prime natural numbers $m,n$, a polynomial has the $mn$-Hartley condition if and only if it has the $m$-Hartley and $n$-Hartley conditions. 
\end{abstract}

\maketitle

\section{The \texorpdfstring{$n$}{n}-Hartley condition}
\label{sec:introduction}

Let $\bar \Q \subset \C$ denote the field of algebraic numbers. Let $n$ be a positive integer. Write $\zeta_n = \exp(2\pi i/n)$, a chosen primitive $n$-th root of $1$ in $\bar \Q$. If $\Delta(t)$ is a polynomial in $\Z[t]$, then we say $\Delta(t)$ has the \emph{$n$-Hartley condition} or is \emph{$n$-Hartley} for short, if there exists a polynomial $g(t) \in \Z[t]$ for which 
\[ \Delta(t^n) = \pm \prod_{j=0}^{n-1} g(\zeta_n^j t). \]
The virtue of the $n$-Hartley condition is this: if a knot $K$ in the $3$-dimensional sphere $S^3$ has a freely periodic symmetry of order $n$, i.e., if the cyclic group of order $n$ acts freely on $S^3$ by orientation-preserving diffeomorphisms that leave $K$ invariant, then the Alexander polynomial $\Delta_K(t)$ of $K$ has the $n$-Hartley condition. This is \cite[Theorem~1.2]{Hartley1981}, slightly modified by \cite[Def.~1.3, Proposition~2.2]{Boyle2026}.

\begin{example}
  Although the $n$-Hartley condition originates in the study of Alexander polynomials of freely periodic knots, the condition is not restricted to such polynomials. We give some further examples here.
  \begin{itemize}
  \item   Trivially, if $f(t) = h(t)^n$ for some integer polynomial $h(t)$, then $f(t)$ has the $n$-Hartley condition: set $g(t) = h(t^n)$.
  \item Products of $n$-Hartley polynomials are again $n$-Hartley.
  \item The $m$-th cyclotomic polynomial $\Phi_m(t)$ has the $n$-Hartley condition for precisely those values of $n$ that are relatively prime to $m$. This is a special case of \cite[Proposition 1.5]{Boyle2026}. 
  \item Other \textit{ad hoc} examples may be found. For instance, $f(t) =t^3-4t^2+3t-1$ is not self-reciprocal and therefore not the Alexander polynomial of a knot. It is, however, a factor of the reducible Alexander polynomial of the knot $10_{75}$ in the tables of \cite{Rolfsen2003}. It satisfies the $3$-Hartley condition with $g(t) = t^3-t^2-1$.
  \item By direct calculation, one sees that the quadratic polynomial
    \[ f_m(t) = t^2 - mt + 1, \quad m \in \ZZ \]
    has the $2$-Hartley condition if and only if one of $m \pm 2$ is a square. Only in two cases is $f_m(t)$ the Alexander polynomial of a knot: $f_1(t)$ and $f_3(t)$. The case of $f_3(t)$ is particularly interesting, since it is proved in \cite[Proposition 1.7]{Hartley1981} that no knot having this Alexander polynomial can be freely periodic of any order $n>1$.

    In particular, not every $n$-Hartley Alexander polynomial is the Alexander polynomial of a freely $n$-periodic knot.
  \end{itemize}  
\end{example}

Recall that a polynomial with integer coefficients is \emph{primitive} if no prime divides all the coefficients. We will say a polynomial in $\Z[t]$ is \emph{irreducible} if it is irreducible as an element of $\Q[t]$. The following is \cite[Proposition~2.8]{Boyle2026}.
\begin{proposition}\label{pr:nHartley}
  Suppose $h(t) \in \Z[t]$ is an irreducible primitive polynomial of positive degree having $\alpha$ as a root and $s$ is a nonnegative integer. The polynomial $\Delta(t)=h(t)^s$ has the $n$-Hartley property if and only if there exists some polynomial $q(t) \in \Q(\alpha)[t]$ of degree $s$ all of whose roots are $n$-th roots of $\alpha$.
\end{proposition}
With this in hand, we see that whether the $n$-Hartley condition holds for $\Delta(t)$ is related to the degrees of irreducible factors of the binomial $t^n -\alpha$. 

\begin{definition}\label{def:lambda}
  Suppose $K$ is a field, $\alpha \in K$ is an element and $n$ is a positive integer. We define $\lambda_K(\alpha, n)$ to be the minimal degree of an irreducible factor of $t^n - \alpha$ over $K$. If $K = \Q(\alpha)$, then we will omit $K$ from the notation, writing $\lambda(\alpha,n)$.
\end{definition}

Our main technical result is this.
\begin{theorem}\label{th:main}
  Let $K$ be a field, let $\alpha \in K$ and let $n$ be a positive integer. If $f(t)$ is a factor of $t^n - \alpha$ over $K$, then \[\lambda_K(\alpha, n) \mid \deg(f(t)).\]
\end{theorem}
The proof of this result is deferred to Section \ref{sec:proof-theorem}.

There are substantial known results concerning the reducibility of binomials $t^n - \alpha$, notably Capelli's Theorem \cite[Theorem~21]{Schinzel1982}, and the results on the nature of the factors of such a binomial, see e.g., \cite[Lemma 3]{Schinzel1977}. We have not found Theorem \ref{th:main} in the literature.

\begin{remark}
  One might hope that Theorem \ref{th:main} is a consequence of inclusions (up to isomorphism) among the fields $K[t]/(f(t))$ as $f(t)$ ranges over the irreducible factors of $t^n-\alpha$. 

  One might conjecture that among all the roots $\beta_1, \dots, \beta_n$ of $t^n - \alpha$, there is a root $\beta$ for which the field $K(\beta)$ appears as a subfield (up to isomorphism) of all the $K(\beta_i)$. Such a $\beta$ would be necessarily a root of a minimal-degree irreducible factor. To see that this conjecture fails, consider $t^8-81$ over $\Q$. We have the factorization into irreducible polynomials
  \begin{equation}\label{eq:1b}
    t^8-81 = (t^2-3)(t^2+3)(t^4+9).
  \end{equation}
  There are two different irreducible factors of minimal degree, and neither of the fields $\Q(\sqrt{3})$ and $\Q(i \sqrt{3})$ may be embedded in the other. This disproves the first conjecture.

  Loosening the condition that $\beta$ should be unique, one might then conjecture that every $K(\beta_i)$ contains some $K(\beta_j)$ where $\beta_j$ is of minimal degree. The factorization in \eqref{eq:1b} again disproves this conjecture: the field $\Q(\zeta_8 \sqrt{3})$ generated by a root of $t^4 + 9$ does not contain either $\Q(\sqrt{3})$ or $\Q(i\sqrt{3})$---its quadratic subfields are $\Q(i), \Q(\sqrt{6})$ and $\Q(i\sqrt{6})$.
\end{remark}

\medskip

The rest of this section proceeds under the assumption that Theorem \ref{th:main} has been established.
Since $\lambda_K(\alpha,n)$ divides the degrees of all factors of $t^n-\alpha$, we see that $\lambda_K(\alpha,n) \mid n$. We prove that $\lambda_K(\alpha, n)$ is multiplicative.
\begin{proposition}\label{pr:multiplicativity}
  Let $K$ be a field of characteristic $0$, let $\alpha \in K$ be an element, and suppose $m,n$ are two relatively prime natural numbers. There is an equality
  \[ \lambda_K(\alpha,m)\lambda_K(\alpha,n) =  \lambda_K(\alpha,mn). \]
\end{proposition}
\begin{proof}
  Let $\beta$ be an $mn$-th root of $\alpha$ of minimal degree. The field extension $K(\beta)/K$ has degree $\lambda_K(\alpha,mn)$. Note that $\beta^m$ and $\beta^n$ are $n$-th and $m$-th roots of $\alpha$, respectively, so that the subextensions $K(\beta^n)/K$ and $K(\beta^m)/K$ have degrees that are multiples of $\lambda_K(\alpha, m)$ and $\lambda_K(\alpha, n)$ by Theorem \ref{th:main}, and divide $[K(\beta):K]$ by elementary field theory. The integers $\lambda_K(\alpha, m),\, \lambda_K(\alpha, n)$ are relatively prime, since they divide $m,\, n$ respectively. Therefore
  \[ \lambda_K(\alpha, m)\lambda_K(\alpha, n) \mid \lambda_K(\alpha,mn). \]

  We finish the proof by showing $\lambda_K(\alpha, mn) \le \lambda_K(\alpha,m)\lambda_K(\alpha, n)$.
  Let $\gamma, \delta$ be $m$-th and $n$-th roots of $\alpha$ of minimal degrees $\lambda_K(\alpha,m)$, $\lambda_K(\alpha,n)$ over $K$. Using  relative primality  again, we deduce that $K(\gamma, \delta)$ has degree $\lambda_K(\alpha,m)\lambda_K(\alpha,n)$ over $K$. We can find some monomial $\gamma^i \delta^j$ (where $i, j \in \Z$) that is an $mn$-th root of $\alpha$, however, so that
  \[ \lambda_K(\alpha, mn) \le \lambda_K(\alpha,m) \lambda_K(\alpha,n), \]
  as required.
\end{proof}

We can restate Proposition \ref{pr:nHartley} as follows:
\begin{proposition}\label{pr:nHartleyBis}
  If $h(t) \in \Z[t]$ is an irreducible primitive polynomial of positive degree having $\alpha$ as a root and $s$ is a nonnegative integer, then $\Delta(t)=h(t)^s$ has the $n$-Hartley property if and only if $\lambda(\alpha, n) \mid s$.
\end{proposition}
\begin{proof}
  Suppose $\Delta(t)$ is $n$-Hartley. A polynomial $q(t)$ exists as laid out in Proposition \ref{pr:nHartley}. Its irreducible factors are factors of $t^n-\alpha$, and so their degrees are multiples of $\lambda(\alpha, n)$. In particular, $\deg(q(t))$ is a sum of such numbers and so is divisible by $\lambda(\alpha,n)$.

  Conversely, if $\lambda(\alpha, n) \mid s$, then we can find a minimal irreducible factor $f(t)$ of $t^n-\alpha$ over $\Q(\alpha)$ whose degree is $\lambda(\alpha,n)$, and by taking some power of $f(t)$, we can produce $q(t)$ as required.
\end{proof}

Our result about multiplicativity of the Hartley condition is a corollary of Proposition \ref{pr:nHartleyBis}.
\begin{theorem}\label{th:multiplicativeHartley}
  Let $m,n$ be relatively prime positive integers. Let $\Delta(t) \in \Z[t]$ be a primitive polynomial. It is $nm$-Hartley if and only if it is both $n$-Hartley and $m$-Hartley.
\end{theorem}
\begin{proof}
  We may factor $\Delta(t)$ as a product of powers of irreducible polynomials $h(t)^s$. Since a polynomial $\Delta(t)$ is $n$-Hartley if and only if each $h(t)^s$ has this property, by \cite[Proposition~2.6]{Boyle2026}, it suffices to handle the case where $\Delta(t) = h(t)^s$. Let $\alpha$ denote a root of $h(t)$. In this case, $\Delta(t)$ satisfies the $n$-Hartley condition if and only if $\lambda(\alpha, n) \mid s$, and similarly for the $m$-Hartley and the $nm$-Hartley conditions.

  Since $\lambda(\alpha, n) \mid n$, and similarly for $m$, the two integers $\lambda(\alpha,n)$, $\lambda(\alpha, m)$ are relatively prime. Therefore
  \[ \lambda(\alpha, n) \mid s \text{ and }  \lambda(\alpha, m) \mid s \qquad \Leftrightarrow \qquad \lambda(\alpha, nm) \mid s, \]
  using Proposition \ref{pr:multiplicativity}. \end{proof}

We conclude this section with remarks on the connection between our results and the theory of symmetric knots. We are especially grateful to Keegan Boyle for explaining this material to us.

For most prime knots $K$, in a sense to be explained below, if $K$ admits free periodicities of relatively prime orders $m$ and $n$, then $K$ admits a free periodicity of order $mn$. This is reflected in our result on polynomials: if the Alexander polynomial of the knot has the $m$- and $n$-Hartley conditions, then it has the $mn$-Hartley condition.

``Most'' here means the prime knots to which \cite[Theorem~1]{Sakuma1986a} applies, and also the torus knots to which it does not apply. A torus knot $T_{p,q}$ can be handled directly: it has free periodicities of all orders $n$ that are relatively prime to $p$ and $q$ by \cite[Exc.~10.2.2(3)]{Kawauchi1996}, and its Alexander polynomial is $n$-Hartley in exactly these cases by an easy application of \cite[Proposition~1.5]{Boyle2026}.

The remaining knots to which \cite[Theorem~1]{Sakuma1986a} does not apply are of two kinds: satellite knots of composite knots, and satellite knots having a companion with respect to which they have trivial winding number. An example showing that there is a knot $K$ having free periodicities of relatively prime orders $m$ and $n$ but no free periodicity of order $mn$ is given in \cite[\S5, Rem.~5.2]{Sakuma1986a}. Take $K$ to be the $(5,1)$-cabling of the connected sum of $4$ copies of the $(13,5)$-torus knot. This is an instance of a satellite knot of a composite knot, so that \cite[Theorem~1]{Sakuma1986a} does not apply. The knot $K$ has free periodicities of orders $2,3,4$ but not of any greater order, in particular, not of order $6$. Nonetheless, the Alexander polynomial $\Delta_K(t)$ has the $2$-Hartley and $3$-Hartley conditions, and by our Theorem \ref{th:multiplicativeHartley}, it has the $6$-Hartley condition.

In fact, the Alexander polynomial of the $K$ above can be calculated explicitly using \cite[Theorem~6.15]{Lickorish2012}. The connected sum of $4$ copies of the $(13,5)$-torus knot has Alexander polynomial $\Phi_{65}(t)^4$, where $\Phi_{n}$ denotes the $n$-th cyclotomic polynomial. Then the $5,1$-cabling gives us an Alexander polynomial
\[ \Delta_K(t) = \Phi_{65}(t^5)^4 = \prod_{\substack{d \mid 5\cdot 65\\ \gcd(d,65)  =1 }} \Phi_{5 \cdot 65/d}(t)^4 = \Phi_{5\cdot 65}(t)^4. \]
This polynomial has the $n$-Hartley condition for all integers $n \ge 2$ that are relatively prime to $65$, by \cite[Proposition~1.5, Proposition~2.6]{Boyle2026}, and in particular, it has the $6$-Hartley condition.

\section{Proof of Theorem {\protect \ref{th:main}}}
\label{sec:proof-theorem}

Throughout, we fix a field $K$ and an element $\alpha \in K^\times$. We will write $\bar K$ for a fixed algebraic closure of $K$. We first discuss the case where $t^n - \alpha$ is not separable, with a view to reducing it to the separable case.
\begin{lemma}\label{lem:charPReduction}
  Suppose $K$ is of characteristic $p > 0$. Let $n$ be a positive integer and write $n=mp^e$ where $e$ is a nonnegative integer and $m$ is a positive integer not divisible by $p$. There is an integer $c$ with $0 \le c \le e$ and a factorization
  \begin{equation}
      \label{eq:1}
      t^n - \alpha =\prod_{i=1}^s h_i(t^{p^{c}})^{p^{e-c}}
    \end{equation}
    where each polynomial $h_i(t)$ is irreducible.

    In this notation $\alpha$ has a $p^{e-c}$-th root $\gamma$ in $K$, and
    \begin{equation}
      \label{eq:11}
      t^{m} - \gamma = \prod_{i=1}^s h_i(t).
    \end{equation}
  \end{lemma}
\begin{proof}
  First we determine the multiplicity of roots of $t^n - \alpha$. The quantity $\alpha$ has $m$ distinct $m$-th roots in $\bar K$. If $\gamma$ is an $m$-th root of $\alpha$, then it has a unique $p^e$-th root in $\bar K$, the polynomial $t^{p^e}-\gamma$ being purely inseparable. Therefore, the roots of $t^n-\alpha$ in $\bar K$ appear with multiplicity $p^e$.

  Next, we describe an arbitrary irreducible factor of $t^n-\alpha$ and its multiplicity. Let $\beta$ be a root of $t^n - \alpha$ and let $g(t)$ be its (monic) minimal polynomial. A standard result on irreducible polynomials (\cite[V, Proposition 6.1]{langAlgebra2002} and its proof) says that $g(t) = h(t^{p^{c}})$ where $h(t)$ is separable and $p^{c}$ is the smallest integer for which $\beta^{p^{c}}$ is separable over $K$. Since $g(t)$ is irreducible over $K$, so too is $h(t)$. The multiplicity of $\beta$ as a root of $h(t^{p^{c}})$ is $p^{c}$ whereas its multiplicity as a root of $t^n - \alpha$ is $p^{e}$, so we deduce that $h(t^{p^{c}})^{p^{e-{c}}}$ divides $t^n-\alpha$, and that the quotient does not have $\beta$ as a root.

  By applying the argument above for roots $\beta_1, \dots, \beta_s$ of different irreducible factors of $t^n-\alpha$, we arrive at the decomposition
  \[ t^n - \alpha = \prod_{i=1}^s h_i(t^{p^{c_i}})^{p^{e-c_i}}, \]
  where the polynomials $h_i(t)$ are irreducible and separable, and for each $i$, the integer $c_i$ is minimal with the property that $\beta_i^{p^{c_i}}$ is separable over $K$.

  We next show that the $c_i$ all agree. Fix some root $\beta = \beta_i$ of $t^n -\alpha$ and set $c=c_i$. Then $p^{c}$ is the smallest power of $p$ for which $\beta^{p^{c}}$ is separable over $K$. By a symmetry argument, it suffices to show that this does not increase if we replace $\beta$ by another $n$-th root of $\alpha$. All such roots are of the form $\beta \zeta$ where $\zeta$ is an $n$-th root of $1$, i.e., an $m$-th root of a $p^e$-th root of $1$. Since $1$ has only $1$ as a $p^e$-th root, we see that $\zeta$ is an $m$-th root of $1$, and in particular is separable. Therefore $(\beta \zeta)^{p^{c}} =\beta^{p^{c}} \zeta^{p^{c}}$ is also separable, which is what we needed to show. This establishes equation \eqref{eq:1}.

  We derive equation \eqref{eq:11}. Let $\beta \in \bar K$ be an $n$-th root of $\alpha$. The constant term of $\prod_{i=1}^s h_i(t^{p^c})$ lies in $K$, and a short calculation shows it is $-\beta^{mp^{c}}$. Setting $\gamma = \beta^{mp^c}$, we have the identity
  \[ (t^{mp^{c}} - \gamma)^{p^{e-c}} = t^n - \alpha  = \Big( \prod_{i=1}^s h_i(t^{p^c}) \Big)^{p^{e-c}}\]
  over $K$. Uniqueness of $p$-th roots in characteristic $p$ then implies equation \eqref{eq:11}.
\end{proof}

In particular, the degrees of the irreducible factors of the polynomial $t^n -\alpha$ are $p^c$ times the degrees of the irreducible factors of the separable polynomial $t^m - \gamma$. Therefore it suffices to prove Theorem \ref{th:main} in the case where $t^n - \alpha$ is separable, i.e., the case in which $n$ is not divisible by the characteristic of $K$. We proceed under this assumption in the rest of this section.
 
Let $\beta$ be a root of $f(t) = t^n - \alpha$ in an algebraic closure $\bar K$. The polynomial $t^n - \alpha$ has roots $S= \{\zeta_n^j\beta\}_{j=0}^{n-1}$. Therefore $F = K(\zeta_n, \beta)$ is a splitting field for $t^n-\alpha$. Let $G$ denote the Galois group of $F$ over $K$. The orbits of the roots of $f(t)$ under the $G$ action correspond to the irreducible factors of $f(t)$ over $K$. We are interested in $\lambda_K(\alpha,n)$, which is the size of the smallest orbit of $G$ acting on $S$. Equivalently we may study the size of the largest stabilizer of $G$ acting on $S$.

The action of $g \in G$ on $K(\zeta_n, \beta)$ is determined by
\[ g(\zeta_n) = \zeta_n^a \qquad g(\beta) = \zeta_n^b \beta. \]
Let us define
\[ U(n) = \left\{
    \begin{bmatrix}
      a & b \\ 0 & 1 
    \end{bmatrix} \, \middle| \, a \in (\Z/n\Z)^\times , \: b \in \Z/n\Z \right\}. \]
It is reasonable to represent the group $G$ as a subgroup of $U(n)$.
We will write an element $\zeta_n^j \beta \in S$ as a column vector $
\begin{bmatrix}
  j \\ 1
\end{bmatrix}$ where $j \in \Z/n\Z$, and use the notation $S(n)$ to denote the set of all such vectors. The action of $g \in G$ on $\zeta_n^j \beta$ is given by ordinary matrix-vector multiplication:
\[ \begin{bmatrix}
  a & b \\ 0 & 1
\end{bmatrix}
\begin{bmatrix}
  j \\ 1
\end{bmatrix} =
\begin{bmatrix}
  aj + b \\ 1 
\end{bmatrix}.
\]
We write $\Orb(G;x)$ to denote the orbit of $x$ under a $G$-action, and $\Stab(G;x)$ to denote the stabilizer subgroup. The cardinality of a set $X$ is denoted $\#X$. We offer a variant of Definition \ref{def:lambda}.
\begin{definition}\label{def:lambdaBis}
  The minimum of the values $\#\Orb(G; \vec v)$ as $\vec v$ ranges over $S(n)$ is denoted $\lambda(G)$.
\end{definition}

We prove the following result, divorced from Galois theory. It implies the separable case of Theorem \ref{th:main} immediately, since the factors of $t^n - \alpha$ correspond to $G$-invariant subsets of $S(n)$.
\begin{proposition}\label{pr:mainAlt}
  Fix a positive integer $n$. Let $G$ be a subgroup of $U(n)$. If $\vec v \in S(n)$, then 
  \[ \lambda(G) \mid \#\Orb(G; \vec v). \]
\end{proposition}
The proof of this follows some lemmas.

\begin{lemma} \label{lem:primePower}
  Proposition \ref{pr:mainAlt} holds if $n$ is a power of a prime.
\end{lemma}
\begin{proof}
  Suppose $n=p^l$ where $p$ is a prime number and $l$ is a positive integer. Using the orbit-stabilizer theorem, we recast the result as saying that the order of every stabilizer of $G$ acting on $S(n)$ divides the order of the largest stabilizer. The result is then trivial if $G$ is a $p$-group, since the orders of all subgroups of $G$ are powers of $p$ and therefore necessarily linearly ordered by divisibility.

  We therefore suppose $\#G$ is not a power of $p$. Since $G$ is a subgroup of $U(n)$ which has order $p^l(p^l - p^{l-1})$, this also implies that $p$ is odd, and therefore that $(\Z/p^l \Z)^\times$ is cyclic. There is a determinant homomorphism $\det: G \to (\Z/p^l\Z)^\times$. The image is a subgroup $Q$ of $(\Z/p^l\Z)^\times$, which is necessarily cyclic. We make two claims: every stabilizer is isomorphic to a subgroup of $Q$, and there exists an element $\vec u$ whose stabilizer is isomorphic to $Q$ itself. This will complete the proof of the lemma.

  Write $H$ for the kernel of $\det : G \to Q$. Every nonidentity element of $H$ acts freely on $S(n)$, since 
  \[
    \begin{bmatrix}
      1 & b \\ 0 & 1 
    \end{bmatrix}
    \begin{bmatrix}
      j \\ 1 
    \end{bmatrix} =
    \begin{bmatrix}
      j +b \\ 1 
    \end{bmatrix}.
  \]
  If $\vec v \in S(n)$ is an element, then the stabilizer subgroup $\Stab(G; \vec v)$ intersects $H$ trivially, and therefore $\Stab(G; \vec v)$ is isomorphic to its image in $Q$ under $\det$. This establishes our first claim.

  Now we produce $\vec u$ whose stabilizer is isomorphic to $Q$, completing the proof. We observe that $H$, the group of matrices in $G$ of the form $\begin{bmatrix}
    1 & b \\ 0 & 1 
  \end{bmatrix}$, is isomorphic to a subgroup of $\Z/p^l\Z$. In particular, $H$ is a cyclic group of power-of-$p$ order.
  
  Since $\#G = \#H \cdot \#Q$ where $G$ is not a $p$-group and $H$ is a $p$-group, we see that $Q$ is not a $p$-group, and consequently is generated by some element $a$ that lies outside of $1 + p\Z/(p^l \Z)$, the maximal $p$-subgroup of $(\Z/p^l\Z)^\times$. Take
  \[ g=
    \begin{bmatrix}
      a & b \\ 0 & 1 
    \end{bmatrix} \in G \subseteq U(n). \]
  Since $\det(g) = a$, the subgroup $\langle g \rangle \subset G$ generated by $g$ maps surjectively to $Q$. The element 
  \[\vec u =
    \begin{bmatrix}
      b/(1-a) \\ 1 
    \end{bmatrix} \in S(n)\]
  is fixed by $g$, so that $\langle g \rangle$ is contained in the stabilizer of $\vec u$. Under the determinant map, this stabilizer maps both injectively and surjectively to $Q$, so it is isomorphic to $Q$ as claimed.
\end{proof}

Our strategy for proving the general case of Proposition \ref{pr:mainAlt} is to use the Chinese remainder theorem to show that the problem is multiplicative, and therefore can be settled one prime at a time.

We begin by considering the effect of reduction modulo divisors of $n$.
Suppose therefore that the positive integer $n$ may be factored $n=mq$ where $m$ and $q$ are relatively prime positive integers. 
Suppose $G \subseteq U(n) \subset \GL(2;\Z/(n))$ is a finite group. Let $G_m, G_q$ denote the images of $G$ under the reduction maps to $\GL(2;\Z/(m))$ and $\GL(2;\Z/(q))$. These are subgroups of $U(m)$ and $U(q)$ respectively. If $g \in G$, we write $g_m, g_q$ for the reductions of $g$. Similarly, if $\vec v \in S(n)$, we write $\vec{v}_m \in S(m)$ and $\vec{v}_q \in S(q)$ for the reductions. We certainly have
\[ (g \cdot \vec v)_m = g_m \cdot \vec v_m \]
since reduction is a ring homomorphism. We allow $G$ to act on $S(m)$ by means of $g\cdot \vec v_m = g_m \cdot \vec v_m$, and similarly for $S(q)$.

The next lemma gives bounds on $\# \Orb(G; \vec{v})$ in terms of the reductions modulo $m$ and $q$.
\begin{lemma}\label{lem:bounds}
  In the notation above, if $\vec v \in S(n)$ is an element, then
  \begin{equation}
    \label{eq:6}
    \lcm\left(\#\Orb(G_m; \vec v_m), \#\Orb(G_q; \vec{v}_q)\right) \,\Big|\, \# \Orb(G; \vec v)
  \end{equation}
  and
  \begin{equation}
    \label{eq:7}
    \#\Orb(G; \vec v) \le \#\Orb(G_m; \vec v_m) \cdot \#\Orb(G_q; \vec{v}_q).
  \end{equation}
\end{lemma}
\begin{proof}
  For a given $\vec v \in S(n)$, the group $\Stab(G; \vec v_m)$ contains $\Stab(G; \vec v)$, and so $\#\Stab(G; \vec v) \mid \# \Stab(G; \vec v_m)$. Since the orbit of the $G$ action on $\vec v_m$ is exactly the orbit of the $G_m$ action, the orbit-stabilizer theorem tells us
  \begin{equation}
    \label{eq:4}
    \left. \# \Orb(G_m; \vec v_m) = [G:  \Stab(G; \vec v_m)] \, \Big| \, [G :\Stab(G; \vec v)] = \# \Orb(G; \vec v)\right..
  \end{equation}
  Since the same applies to reduction modulo $q$, divisibility relation \eqref{eq:6} follows.

  We have $g \cdot \vec v = \vec v$ if and only if $g \cdot \vec v_m = \vec v_m$ and $g \cdot \vec v_q = \vec v_q$, using the Chinese remainder theorem. That is to say $\Stab(G; \vec v) = \Stab(G; \vec v_m) \cap \Stab(G; \vec v_q)$. Then by elementary group theory
  \begin{equation}
    \label{eq:5}
    \# \Orb(G; \vec v) = [G : \Stab(G; \vec v_m) \cap \Stab(G; \vec v_q)] \le [G : \Stab(G; \vec v_m)][G : \Stab(G; \vec v_q)] =  \# \Orb(G_m; \vec v_m)\ \# \Orb(G_q; \vec v_q),
  \end{equation}
  which is inequality \eqref{eq:7}.  
\end{proof}

The next lemma is used to show that the $\lcm$ in divisibility relation \eqref{eq:6} may be replaced by a product, in the case of minimal orbits.
\begin{lemma}\label{lem:divisibilityLemma}
  Suppose Proposition \ref{pr:mainAlt} is known to hold for some $n$ and $G$. Then $\lambda(G) \mid n$.
\end{lemma}
\begin{proof}
  The $n$-element set $S(n)$ decomposes as a disjoint union of orbits, all of whose cardinalities are multiples of $\lambda(G)$.
\end{proof}

We now complete the proof.
\begin{proof}[Proof of Proposition \ref{pr:mainAlt}]
  We argue by induction on the number of distinct prime factors of $n$. The base case, where $n$ is a power of a prime, is handled by Lemma \ref{lem:primePower}.

  If $n$ is not a prime power, write $n=mq$ where $m$ and $q$ are relatively prime and each has fewer distinct prime factors than $n$.

  For an arbitrary $\vec{v} \in S(n)$, we use Lemma \ref{lem:bounds} and the induction hypothesis to show
  \begin{equation*}
    \label{eq:9}
    \lcm\left(\lambda(G_m), \lambda(G_q)\right) \, \Big|\, \lcm\left(\#\Orb(G_m; \vec{v}_m), \#\Orb(G_q; \vec{v}_q)\right) \, \Big|\, \#\Orb(G; \vec{v}).
  \end{equation*}
  Since $m$ and $q$ are relatively prime, Lemma \ref{lem:divisibilityLemma} assures us that the least common multiple of $\lambda(G_m)$ and $\lambda(G_q)$ is actually the product, so we deduce
  \begin{equation}
    \label{eq:10}
    \lambda(G_m) \lambda(G_q) \mid \# \Orb(G; \vec{v}).
  \end{equation}
  
  Since $\vec v$ is arbitrary, the result will be proved once we establish
  \begin{equation}
    \label{eq:8}
    \lambda(G_m)\lambda(G_q) = \lambda(G),
  \end{equation}
  i.e., that this lower bound is attained.

  Take $\vec w \in S(n)$ to be an element whose reductions $\vec{w}_m \in S(m)$ and $\vec{w}_q \in S(q)$ have $G$-orbits as small as possible, that is, they are of cardinalities $\lambda(G_m)$ and $\lambda(G_q)$. Such a $\vec{w}$ exists by virtue of the Chinese remainder theorem, which implies $S(n) \to S(m) \times S(q)$ is a bijection. By Lemma \ref{lem:bounds}, we deduce $\# \Orb(G; \vec w) \le \lambda(G_m) \cdot \lambda(G_q)$. By applying \eqref{eq:10} to $\vec{w}$ we see that
  \[ \lambda(G_m) \lambda(G_q) \mid \# \Orb(G; \vec{w}) \le \lambda(G_m) \cdot \lambda(G_q). \]
  This shows that \eqref{eq:8} holds, completing the proof.
\end{proof}

\begin{remark} \label{rem:allGalois}
  For all natural numbers $n> 1$, we can find infinitely many $\alpha \in \Q$ for which the Galois group $G$ of $\Q(\zeta_n, \beta)$ over $\Q$ is $U(n)$, i.e., as big as possible. Throughout this remark, $\beta$ is used to denote an $n$-th root of $\alpha$. To see that infinitely many such fields exist, we argue as follows.

  Suppose $\alpha$ is a rational prime that does not divide $n$ and $l> 1$ is a divisor of $n$. The field $\Q(\sqrt[l]{\alpha})$ has degree $l$ over $\Q$ and $\alpha$ is the unique prime not dividing $n$ that ramifies in this field. Therefore another pair $(\alpha', l') \neq (\alpha, l)$ satisfying the same hypotheses gives rise to a distinct field $\Q(\sqrt[l']{\alpha'})$.

  There being only finitely many subfields of $\Q(\zeta_n)$, but infinitely many pairs $(\alpha, l)$, we conclude that there must be infinitely many prime numbers $\alpha$ for which none of the fields $\Q(\sqrt[l]{\alpha})$ are subfields of $\Q(\zeta_n)$.
  
  It follows from the results of \cite{darbiSullaRiducibilitaEquazioni1927} (restated in English as the first theorem of \cite{gayDegreeSplittingField1978}) that $\Q(\zeta_n, \beta)$ has degree $n\phi(n)$ over $\Q$, and therefore has Galois group isomorphic to the full group $U(n)$.

  It is a consequence of this that, for any natural number $n$ and any subgroup $G \subset U(n)$, one can find an algebraic number field $K$ and a polynomial $t^n - \alpha$ over $K$ for which the Galois group of $K(\zeta_n, \beta)$ over $K$ is isomorphic to $G$.

  We do not know whether, given $G$, one can necessarily find an algebraic integer $\alpha$ for which $\Q(\beta, \zeta_n)/\Q(\alpha)$ has Galois group $G$, although we see no reason why not. We also see no reason to suppose $\alpha$ could not be chosen to be a root of an Alexander polynomial of some knot.
\end{remark}


\section*{Acknowledgements}\label{sec:acknowledgements}

The authors thank the founders and maintainers of the website \texttt{\href{https://mathoverflow.net/}{MathOverflow}}. The main argument in this paper first appeared on that website as an answer by the first author to a question posed by the second, \cite{Bolan2025o}.

We thank Keegan Boyle for comments on a draft of this manuscript. Ben Williams also thanks Keegan Boyle for introducing him to the $n$-Hartley condition, and for a great many conversations on the topic of symmetric knots.

We thank an anonymous referee for helpful comments which have improved the paper in several ways, and in particular for prompting us to consider the inseparable case in Section \ref{sec:proof-theorem}.

\section*{Disclosures}
\label{sec:disclosures}

Large language models were used to proofread this paper. Any mistakes that remain are the responsibility of the authors.

\printbibliography
\end{document}
